\documentclass[11pt]{amsart}
\usepackage{amssymb,hyperref}
\usepackage{graphicx}
\usepackage[matrix,arrow]{xy}
\newcommand{\PP}{\mathbb P}

\newcommand{\CC}{{\mathbb C}}
\newcommand{\NN}{{\mathbb N}}
\newcommand{\ZZ}{{\mathbb Z}}
\newcommand{\QQ}{{\mathbb Q}}
\newcommand{\RR}{{\mathbb R}}

\newcommand{\Pic}{\mathrm{Pic}}

\newcommand{\tD}{2h}
\newcommand{\jD}{h}
\newcommand{\polD}{H}

\DeclareMathOperator{\rank}{rank}

\numberwithin{equation}{section}
\begin{document}

\title{24 rational curves on K3 surfaces}

\author{S\l awomir Rams}
\address{Institute of Mathematics, Jagiellonian University, 
ul. {\L}ojasiewicza 6,  30-348 Krak\'ow, Poland}
\email{slawomir.rams@uj.edu.pl}

\author{Matthias Sch\"utt}
\address{Institut f\"ur Algebraische Geometrie, Leibniz Universit\"at
  Hannover, Welfengarten 1, 30167 Hannover, Germany}

    \address{Riemann Center for Geometry and Physics, Leibniz Universit\"at
  Hannover, Appelstrasse 2, 30167 Hannover, Germany}

\email{schuett@math.uni-hannover.de}

\date{March 4, 2022}
\thanks{Research partially supported by the National Science Centre, Poland, Opus  grant 
no.\ 2017/25/B/ST1/00853
(S.\ Rams)}

\begin{abstract}
Given $d\in\NN$, we prove that all 
smooth K3 surfaces (over any field of characteristic $p \neq 2,3$) of degree greater than $84d^2$ 
contain at most 24 rational curves of degree at most $d$.
In the exceptional characteristics, the same bounds hold for non-unirational K3 surfaces,
and we develop analogous results in the unirational case.
For $d \geq 3$, we also construct  K3 surfaces
of any degree greater than $4d(d+1)$
with 24 rational curves of degree exactly $d$,
thus attaining the above bounds.
\end{abstract}

\maketitle

\newcommand{\XXd}{X_{d}}
\newcommand{\XXf}{X_{4}}
\newcommand{\XXp}{X_{5}}
\newcommand{\mF}{\mathcal F}
\newcommand{\MW}{\mathop{\mathrm{MW}}}
\newcommand{\mL}{\mathcal L}
\newcommand{\mR}{\mathcal R}
\newcommand{\Ruledeight}{S_{11}}
\newcommand{\Ruledfour}{S_{4}}
\newcommand{\DivisorRest}{\mathfrak Rest}
\newcommand{\Pl}{\Pi}
\newcommand{\reg}{\operatorname{reg}}

\newcommand{\IK}{{{\rm I}}}
\newcommand{\II}{{\mathop{\rm II}}}
\newcommand{\III}{{\mathop{\rm III}}}
\newcommand{\IV}{{\mathop{\rm IV}}}
\newcommand{\transpose}{T}  

\theoremstyle{remark}
\newtheorem{obs}{Observation}[section]
\newtheorem{rem}[obs]{Remark}
\newtheorem{example}[obs]{Example}
\newtheorem{ex}[obs]{Example}
\newtheorem{conv}[obs]{Convention}
\theoremstyle{definition}
\newtheorem{Definition}[obs]{Definition}
\theoremstyle{plain}
\newtheorem{prop}[obs]{Proposition}
\newtheorem{theo}[obs]{Theorem}
\newtheorem{Theorem}[obs]{Theorem}
\newtheorem{lemm}[obs]{Lemma}
\newtheorem{crit}[obs]{Criterion}
\newtheorem{claim}[obs]{Claim}
\newtheorem{Fact}[obs]{Fact}
\newtheorem{cor}[obs]{Corollary}
\newtheorem{assumption}[obs]{Assumption}

\newcommand{\ux}{\underline{x}}
\newcommand{\ud}{\underline{d}}
\newcommand{\ue}{\underline{e}}
\newcommand{\mmS}{{\mathcal S}}
\newcommand{\mmP}{{\mathcal P}}
\newcommand{\nlines}{\mbox{\texttt l}(\XXp)}
\newcommand{\ii}{\operatorname{i}}

\newcommand{\nonlinflec}{{\mathcal D}}
\newcommand{\linflec}{{\mathcal L}}
\newcommand{\flec}{{\mathcal F}}


\section{Introduction}
\label{intro}

The study of rational curves on projective K3 surfaces has a long history,
starting with the result of Bogomolov and Mumford
that every complex projective K3 surface contains a (possibly singular) rational curve \cite{MM}.
The conjecture that every K3 surface over an algebraically closed field
 contains infinitely many rational curves,
was recently proven  in characteristic zero   in \cite{CGL}, \cite{CGL2},
building on previous work in 
%
%
 \cite{BT},
 \cite{BHT}, \cite{LL}.

%
%

The problem of rational curves assumes a different flavour when we
consider 
polarized K3 surfaces of a fixed degree $\tD$ (i.e. pairs $(X,H)$, such that $H \in \mbox{Pic}(X)$ is very ample with  $H^2=\tD$) 
and take the degrees of the rational curves
relative to the  polarization $H$ into account.
Denote 
\[
r_d := r_d(X) := \#\{\text{rational curves } C\subset X \text{ with } \deg(C)= d = C.H\},
\]
For surfaces of small degree, the behaviour of  $r_d$, especially its maximum, seems to be hard to predict in general,
although the problem has a long history (cf.\ \cite{degt}, \cite{DIS}, \cite{RS}, \cite{Segre}).
In contrast, for complex K3 surfaces of high degree $\tD$ (i.e.\ $h>2d^2$),
Miyaoka  \cite{Miyaoka} applied the orbibundle Miyaoka--Yau--Sakai inequality from \cite{Miyaoka-orbi}
to obtain the following  bound: 
\begin{eqnarray}
\label{eq:Miyaoka}
\frac{1}{d} r_1+ \frac{2}{d} r_2+\hdots+ r_d \leq\frac{24 \jD}{	\jD-2d^2}.
\end{eqnarray}
In particular, this implies that for $\jD>50d^2$,  
one has 
\[
r_i\leq 24 \;\;\;\; \forall \,i\leq d,
\]
but it remained open to what extent this bound is sharp 
and which configurations of rational curves attain the maximal values \cite[Rem.~(2)]{Miyaoka}.
Here we remove the weights,
i.e. we consider the numbers 
\[
S_d := r_1+\hdots+r_d = \#\{\text{rational curves } C\subset X \text{ with } \deg(C)\leq d\},
\]
and use lattice theory to obtain characteristic-free bounds 
and characterize the K3 surfaces attaining them.
In particular, we show that the bounds are sharp 
(already for smaller $h$).


\begin{Theorem}
\label{thm}
Let $d\in\NN$.
\begin{enumerate}
\item[(i)]
For all $\jD>42d^2$ and for all 
 K3 surfaces $X$ of degree $\tD$ over a
 field $k$ of characteristic $p\neq  2,3$,
one has
\[
S_d\leq 24.
\]
\item[(ii)]
If $\jD\gg 0$ and $S_d>21$, 
then the  rational curves of degree at most $d$ are fibre components
of a genus one fibration.
\item[(iii)]
For $d\geq 3$ and for all $\jD\geq d(2d+1)-1$,
there are K3 surfaces of degree $2h$ with $r_d=24$
over fields of characteristic $\neq 2$.
\end{enumerate}
\end{Theorem}

%

In characteristic $2, 3$, due to the presence of quasi-elliptic fibrations
we have slightly weaker bounds
involving the following restricted rational curve count:
\[
S_d' = S_d'(X) :=
\#\left\{\begin{matrix}
\text{rational curves } C\subset X \text{ with } \deg(C)\leq d\\
\text{such that $p_a(C)\neq 1$ if $C$ is cuspidal}
\end{matrix}
\right\}.
\]

\begin{Theorem}
\label{thm2}
Let $d\in\NN$.
For all 
K3 surfaces $X$ of degree $\tD$ over a
 field $k$ of characteristic $p= 2$,
one has:
\begin{enumerate}
\item[(i)]
if $X$ is not unirational and $\jD>42d^2$, then 
$
S_d \leq
24.
$
\item[(ii)]
in general, if $\jD>46.25d^2$, then
$
S_d' \leq
40.
$
\end{enumerate}
\end{Theorem}

For infinitely many $d$, there are K3 surfaces of degree $2h$
with $r_d=24$ resp.\ $40$ for infinitely many integers $\jD$ each
over fields of characteristic $2$
(see Remark \ref{rem:24-2} and Section \ref{ss:2}).

\begin{Theorem}
\label{thm3}
Let $d\in\NN$.
For all 
K3 surfaces $X$ of degree $\tD$ over a
field $k$ of characteristic $p=3$,
one has:
\begin{enumerate}
\item[(i)]
if $\jD>42d^2$ and $X$ is not unirational or $X$ has Artin invariant $\sigma>6$, then 
$S_d\leq 24$;
\item[(ii)]
in general, if $\jD>43d^2$, then
$S_d'\leq 30$.
\end{enumerate}
\end{Theorem}

For all $d$, there are K3 surfaces of degree $2h$ 
with $r_d=30$ over fields of characteristic $3$
for infinitely many $\jD$ (see \ref{ss:3}).

\begin{rem}
\label{rem}
Case {(ii)} in Theorems \ref{thm2}, \ref{thm3} applies to all rational curves
if $d\leq 2$ (as these curves are automatically smooth, we have $S_d=S_d'$).
For $d=1$ and $p=2$, the estimate from Theorem \ref{thm2} {(ii)} can be improved to  
\begin{equation} \label{eq-25}
r_1 = \#\{\text{lines on } X\} \leq 25.
\end{equation}
\end{rem}

%
%
%
%
%

Of course, all bounds from Theorems~\ref{thm},~\ref{thm2},~\ref{thm3}
also apply to smooth rational curves of given degree. 
It may come as a surprise that even under this restriction,
they are attained infinitely often -- even when we consider only 
smooth rational  curves of degree exactly $d$ (see Section \ref{s:smooth}).

Our results also relate to work of Degtyarev on complex K3 surfaces.
In the paper \cite{degt} which inspired our approach, Degtyarev considers lines only,
but his methods 
give precise maxima for the number of lines depending on the degree $\tD$.
In particular, Degtyarev shows that  there exist arbitrary large $h \in \NN$ such that 
the bound from Theorem \ref{thm}
is never attained by lines on K3 surfaces of degree $\tD$ over $\CC$.
A similar pattern persists in positive characteristic (and also for $d=2$),
but the precise analysis exceeds the scope of this paper.
In contrast to the case of lines and conics,
rational curves of degree $3$ and higher exhibit a very regular behaviour
as we show in this paper (also over fields of positive characteristic).




\begin{rem}
Analogous techniques apply to Enriques surfaces,
see \cite{RS-12}.
\end{rem}

\begin{conv}
Since Theorems \ref{thm}--\ref{thm3} stay valid under base extension,
we will assume without loss of generality that the base field $k$ of characteristic $p\geq 0$
is algebraically closed.
All rational curves are assumed to be irreducible.
\end{conv}


\section{Set-up}

We consider polarized K3 surfaces of degree $\tD$, i.e. pairs $(X,H)$ where $X$ is a smooth K3 surface  
over an algebraically closed field $k$ of characteristic $p$, and  $H$ 
is a very ample divisor of square $H^2=\tD$.
Then, the linear system $|H|$ defines an embedding
\[
X \hookrightarrow \PP^{\jD+1}
\]
which is an isomorphism onto its image
(the latter is not contained in any hyperplane of $\PP^{\jD+1}$). 

Conversely, one can check whether a given divisor $H$ on $X$ with $H^2=\tD>0$ is very ample
by the methods developed in \cite{S-D}.
In detail, 
one has:

\begin{crit}
\label{crit}
Let $p \neq 2$.
A divisor $H$ on a K3 surface $X$ 
is very ample if
\begin{enumerate}
\item
$H.C>0$ for every curve $C\subset X$;
\item
$H.E>2$ for every irreducible curve $E\subset X$ of arithmetic genus $1$;
\item
$H^2\geq 4$, and if $H^2=8$, then 
$H$ is not $2$-divisible in $\Pic(X)$.
\end{enumerate}
\end{crit}

We will revisit this criterion in Section \ref{s:ex}
in order to apply it to certain divisors on  K3 surfaces with genus one fibrations.
Then any irreducible curve serves  either as a fibre component $\Theta$ 
or as a multisection $D'$ of index $d'=F.D'$
where $F$ denotes any fibre.
This structure simplifies the analysis of assumptions of Criterion \ref{crit} substantially,
especially since $p_a(D')>0$ implies $d'>1$.

%
%
%


\section{Preparations}
\label{s:prep}

Given a K3 surface $X$ of degree $2h$, 
consider the set
\[
\Gamma = \{\text{rational curves } C\subset X \text{ with } \deg(C)\leq d\}
\]
which we can also interpret as a graph without loops
with multiple edges corresponding to the intersection number $C.C'$ for $C, C'\in\Gamma$.
For ease of exposition, we say that an effective divisor $D$ is \emph{supported on $\Gamma$}
if there is a subgraph $\Gamma'\subset \Gamma$ such that
$$
\mbox{supp}(D)  = \bigcup_{C \in \Gamma'} C \, .
$$
For each curve (or vertex) $C\in\Gamma$, we take two values into consideration: 
the square $C^2$ 
and the degree $d_C=C.H$ of the corresponding curve. 

Together these rational curves generate the formal group 
$$M:=\ZZ\Gamma\subset\mathrm{Div}(X),
$$
equipped with the intersection pairing extending linearly that on the single curves.
We emphasize that $M$ may be degenerate 
(in shorthand $\ker(M)\neq 0$) as it falls into the following three cases:
\begin{enumerate}
\item
elliptic -- $M\otimes\RR$ is negative-definite ($\ker(M)=0$);
\item
parabolic -- $M\otimes\RR$ is negative semi-definite, but not elliptic ($\ker(M)\neq 0$);
\item
hyperbolic -- $M\otimes\RR$ has a one-dimensional positive-definite subspace
and none of greater dimension.
\end{enumerate}
The last condition comes from the Hodge index theorem
which will also enter crucially in the next sections.

The squares $C^2\geq -2$ of the curves in $\Gamma$ play a fundamental role 
in distinguishing these three cases. 
First we discuss the elliptic case.

%
%

\section{Elliptic case}
\label{s:ell}

The results in this section and the next will be independent of $h$,
so they can be used to construct K3 surfaces with up to 24 rational curves
for arbitrary degrees (see Section \ref{s:ex}).

If $\Gamma$ is elliptic, then clearly $C^2=-2$ for any $C\in\Gamma$.
Moreover, if $C.C'> 0$ for $C, C'\in \Gamma$, then
$C.C'=1$ for else $(C+C')^2\geq 0$, and $\Gamma$ would not be elliptic.
In summary,
\begin{eqnarray}
\label{eq:01}
C.C'\in \{0,1\} \;\;\; \text{ for all } \;\; C\neq C'\in\Gamma.
\end{eqnarray}

\begin{lemm}
\label{lem:Dynkin}
If $M$ is elliptic, then it is an orthogonal sum of finitely many Dynkin diagrams (ADE-type).
\end{lemm}

\begin{proof} 
By assumption, the lattice $M$ is negative-definite and non-degenerate.
Since it embeds into $\Pic(X)$ which is hyperbolic of rank $\rho(X)\leq 22$
(or $\rho(X)\leq 20$ if $p=0$, see Remark \ref{rem-elliptic19}),
we find that 
\begin{eqnarray}
\label{eq:rk21}
\rank M\leq 21.
\end{eqnarray}
The claim now  follows from \eqref{eq:01}
and the classification of (negative-definite) root lattices. 
\end{proof}

We can now derive the first  step towards Theorem \ref{thm}:

\begin{cor}
\label{cor:21}
If $M$ is elliptic, then $\#\Gamma \leq \rank M \leq 21$.
\end{cor}

\begin{proof}
Since $M$ is elliptic we have $\#\Gamma = \rank M$. Thus the claim follows from \eqref{eq:rk21}.
\end{proof}

\begin{rem} \label{rem-elliptic19}
Equality 
can only be attained in characteristics $p\leq 19$ by \cite{Shimada}.
Over $\CC$, one even has the bound $\#\Gamma\leq 19$, since $\rho(X)\leq 20$  by Lefschetz' $(1,1)$-theorem,
equality being attained, for instance, by  extremal elliptic K3 surfaces
(cf.\ \cite{MP}, 
\cite{Shimada-Zhang}). The condition in Theorem \ref{thm} (ii) can thus be improved to $S_d>19$.
\end{rem}


\section{Parabolic case}
\label{s:para}

If $\Gamma$ is parabolic, then this implies as before that $C^2\leq 0$ for all $C\in\Gamma$.
Moreover, for any isotropic divisor $D\in M$, we obtain 
\begin{eqnarray}
\label{eq:isoD}
D.C=0 \;\;\; \text{ for all } \;\;\; C\in\Gamma,
\end{eqnarray}
since else $(2D\pm C)^2>0$. In particular, this applies to all curves $C\in\Gamma$ with $C^2=0$:
$C\in\ker(M)$.
All other curves $C\in\Gamma$ satisfy $C^2=-2$.
Arguing as before, we find 
\begin{eqnarray}
\label{eq:012}
C.C' \in\{0,1,2\} \;\; \text{ for all } \;\; C\neq C'\in\Gamma. 
\end{eqnarray}

\begin{lemm}
\label{lem:ortho}
If $M$ is parabolic, then it is an orthogonal sum of finitely many Dynkin diagrams, finitely many extended Dynkin diagrams ($\tilde A\tilde D\tilde E$-type) and isotropic vertices. Moreover,  at least one summand is either an isotropic vertex or an extended Dynkin diagram. 
\end{lemm}

\begin{rem}
In the parabolic case, $\Gamma$ may a priori involve infinitely many isotropic vertices,
cf. the quasi-elliptic case in Section \ref{s:exc}.
\end{rem}

\begin{proof} 
%
Assume that $M$ is parabolic. If there is $C\in\Gamma$ with $C^2=0$, then $C\in\ker(M)$ as noticed above.
This leads to an orthogonal decomposition 
$M = M_z \oplus M_{r}$
with all vertices $C \in \Gamma \cap M_{z}$ satisfying the condition $C^2 = 0$,
and $\Gamma \cap M_{r}$  comprising exactly the roots. Moreover, we have
$C.C'\in\{0,1,2\}$ for any $C\neq C'\in\Gamma \cap M_{r}$ by \eqref{eq:012}.

If $C, C' \in  \Gamma \cap M_{r}$ satisfy $C.C'=2$, then $D=C+C' \in M_{r}$ is isotropic
 and we get the extended Dynkin diagram $\tilde A_1$, that defines an orthogonal summand of $M_{r}$
 by \eqref{eq:isoD}.
Since each such summand gives $A_1\hookrightarrow \Pic(X)$, 
we can find only finitely many such summands in $M_{r}$.
Away from these summands, we may assume that 
$C.C'\in\{0,1\}$ for any $C\neq C'\in\Gamma \cap M_{r}$ by \eqref{eq:012}. 
The classification of (negative semi-definite) root lattices 
shows that $M_r$ further decomposes into (extended) Dynkin diagrams as claimed.
Since $M/\ker(M)$ gives a negative-definite sublattice of $\Pic(X)$, both 
the number and the rank of these summands are bounded.
\end{proof}

\begin{lemm}
\label{lem:para}
If $M$ is parabolic, then there is a genus one fibration
\begin{eqnarray}
\label{eq:fibr}
X \to \PP^1
\end{eqnarray}
such that
\[
\Gamma =  \{\text{rational fibre components $\Theta$ of \eqref{eq:fibr} with } \deg(\Theta)\leq d\}.
\]
\end{lemm}

\begin{proof}
By Lemma \ref{lem:ortho}, $M$ contains an isotropic vertex $v_0 \in \Gamma$ or an extended Dynkin diagram.
Either yields an isotropic divisor $D$ of Kodaira type,
i.e.\ a nodal or cuspidal curve of arithmetic genus one
or a configuration of smooth rational curves (with multiplicities) which appears in 
Kodaira's and Tate's classification of singular fibres of elliptic surfaces \cite{K}, \cite{Tate}.
The linear system $|D|$ gives the claimed genus one fibration \eqref{eq:fibr}.
By construction, $D$ is a fibre of \eqref{eq:fibr},
and all vertices of $\Gamma$ feature as fibre components
of the fibration by \eqref{eq:isoD}.
Conversely, any rational component $\Theta$ of a singular fibre
of degree at most $d$ appears in $\Gamma$
while multisections would force $\Gamma$ to be hyperbolic.
\end{proof}

We shall now discuss how Lemma \ref{lem:para} fits together
with the bound from Theorem \ref{thm} (i).
If the general fibre of the genus one fibration  \eqref{eq:fibr} is smooth,
then the number of rational fibre components
can  be bounded by topological arguments.
Indeed, in terms of the Euler--Poincar\'e characteristic, we have
\begin{eqnarray}
\label{eq:e=24}
24 = e(X) = \sum_{t\in\PP^1} e(F_t) + \delta_t.
\end{eqnarray}
Here $\delta_t\geq 0$ accounts for the wild ramification
(which only occurs for certain additive fibre types in characteristics $2$ and $3$, see
\cite{SSc}).
Moreover, for a singular fibre $F_t$ with $m_t$ irreducible components, one has
\begin{eqnarray}
\label{eq:comps}
\;\;\;
\hspace{.5cm}
e(F_t) =
\begin{cases}
m_t, &
\text{ if } F_t \text{ is multiplicative (type $\IK_n, n>0$)},\\
m_t+1, &
\text{ if } F_t \text{ is additive}.
\end{cases}
\end{eqnarray}
%
By \cite{Tate-genus}, the general fibre can be non-smooth only in characteristics $2,3$,
so outside those characteristics we have
\begin{eqnarray}
\label{eq:<=24}
\#\Gamma \leq \{\text{rational fibre components of \eqref{eq:fibr}}\} \leq 24.
\end{eqnarray}
This settles Theorem \ref{thm} (i) in the parabolic case.

In 
characteristics $2$ and $3$,
the genus one fibration \eqref{eq:fibr} may also be quasi-elliptic,
i.e. the general fibre is a cuspidal cubic
(as excluded in the curve count $S_d'$).
Note that this automatically implies the surface to be unirational,
so if $X$ is not unirational,
Theorem \ref{thm2} (i) follows from what we have seen above.
Moreover, quasi-elliptic K3 surfaces in characteristic $3$ automatically satisfy $\sigma\leq 6$
(just compute the sublattice generated by fibre components and the curve of cusps),
so Theorem \ref{thm3} (i) follows as well.

To conclude the proofs, let \eqref{eq:fibr} be quasi-elliptic,
but remove all cuspidal rational curves with $p_a=1$ from consideration.
Hence Lemma \ref{lem:para} leads
to
\[
\Gamma' =   \{\text{smooth rational fibre components $\Theta$ of \eqref{eq:fibr} with } \deg(\Theta)\leq d\}.
\]
That is, we only have to inspect the reducible fibres.
Since the general fibre $F$ has $e(F)=2$, the formula for the Euler-Poincar\'e characteristic 
of a quasi-elliptic fibration reads
\begin{eqnarray}
\label{eq:e=24'}
24 = e(X) = 4 + \sum_{t\in\PP^1} (e(F_t)-2).
\end{eqnarray}
Since all fibres are additive, \eqref{eq:comps} gives 
\[
\# \Gamma' \leq  20 + \#\{\text{reducible fibres}\}.
\]
By \eqref{eq:e=24'}, there are at most 20 reducible fibres,
so the estimate from Theorem \ref{thm2} (ii) for $p=2$ follows readily (and is attained by
the fibrations from \cite{RuS},
see \ref{ss:2}).
If $p=3$, then the only possible reducible fibre types of a quasi-elliptic fibration are $\IV, \IV^*, \II^*$,
so $e(F_t)-2\geq 2$ for a reducible fibre,
and \eqref{eq:e=24'} allows for at most 10 of them.
Hence $\#\Gamma'\leq 30$ as stated in Theorem \ref{thm3} (ii).
To see that the bound is attained (for certain $\jD$), confer \ref{ss:3}.


%
%
%


\section{Intrinsic polarization}

It remains to study the case where $M$ is hyperbolic
 and where finally  the degree of the K3 surface $X$ plays a role. 
Extending on \cite{degt}, we 
consider the hyperbolic lattice
\[
L = M/\ker(M)
\]
and
try to equip $L\otimes\QQ$ with an intrinsic polarization $H_\Gamma\in L\otimes\QQ$
obtained by solving for 
$$
C.H_\Gamma=d_C \;\;\;\forall \, C\in\Gamma.
$$
Note that $H_\Gamma$ need not exist at all, since the system of equations tends to be overdetermined.
The intrinsic polarization has to be compared to the canonical way of enhancing $M$
by the 
polarization $H$ with $H^2=\tD$
where we set 
$$
C.H=d_C \;\;\;\forall \, C\in\Gamma.
$$
This leads to the non-degenerate lattice
$
L_{\polD} = (\ZZ\Gamma + \ZZ H)/\ker(\ZZ\Gamma + \ZZ H).
$

\begin{lemm}
\label{lem:embed}
If $L_{\polD}$ is hyperbolic, then $L$ embeds into $L_{\polD}$.
\end{lemm}

\begin{proof}
It suffices to verify that 
\[
\ker(\ZZ\Gamma)\subset\ker(\ZZ\Gamma + \ZZ H).
\]
Assume to the contrary that there is some $w\in\ker(\ZZ\Gamma)\setminus\ker(\ZZ\Gamma + \ZZ H)$.
Picking some positive vector $x\in\ZZ\Gamma$, we obtain an auxiliary rank 3 lattice $L'=\ZZ w + \ZZ x + \ZZ H$
with intersection form
\[
Q = 
\begin{pmatrix}
0 & 0 & w.H\\
0 & x^2 & x.H\\
w.H & x.H & H^2
\end{pmatrix}
\]
This has determinant $\det(Q)=-x^2(w.H)^2<0$.
On the other hand, $L'$ is hyperbolic by the Hodge index theorem,
so $\det(Q)>0$ gives the required contradiction.
\end{proof}

%
%

We are now in the position to formulate and prove the following key reduction result:

\begin{prop}
\label{prop:h^2}
If $L_{\polD}$ is hyperbolic, then $H_\Gamma$ exists and $\tD \leq H_\Gamma^2$.
\end{prop}

The proof of Proposition \ref{prop:h^2} follows the ideas of \cite{degt}
(which also states the converse implication).
For completeness, we provide a direct argument.

\begin{proof}
By Lemma \ref{lem:embed}, the lattice $L$ embeds into $L_{\polD}$ with corank 0 or 1.
The underlying vector spaces thus admit an orthogonal decomposition
\[
L_{\polD}\otimes\QQ = (L\otimes\QQ) \perp (L^\perp\otimes\QQ).
\]
Here $L^\perp$ is either zero or negative-definite,
since both lattices $L$ and $L_{\polD}$ are hyperbolic by assumption.
Express $H$ uniquely as
\[
H = H_\Gamma + H_\Gamma^\perp, \mbox{ where } H_\Gamma\in L\otimes\QQ, \;\; H_\Gamma^\perp\in L^\perp\otimes\QQ.
\]
Hence $H_\Gamma$ exists, 
and
\begin{equation} \label{ineq-orthdecomp}
H^2 = H_\Gamma^2 + (H_\Gamma^\perp)^2 \leq H_\Gamma^2
\end{equation}
as stated.
\end{proof}

\begin{rem}
\label{rem:Gamma_0}
The above arguments also apply to any hyperbolic subgraph $\Gamma_0\subset\Gamma$.
This will be used in the sequel, 
for instance in Example \ref{ex:D_6}.
\end{rem}


Proposition \ref{prop:h^2} forms a cornerstone of our argument
due to the following consequence:

\begin{cor}
\label{cor:finite}
If
 $\Gamma$ is hyperbolic, then it
 can be realized by rational curves on K3 surfaces of degree $\tD$  only for a finite number of integers $\jD$.
 \end{cor}
 
 \subsection*{Proof of Theorem \ref{thm} (ii)}
 
 Assume that $\#\Gamma>21$.
 By Corollary \ref{cor:21}, $\Gamma$ cannot be elliptic.
We claim that it is not hyperbolic, either.
Otherwise, there is some elliptic or parabolic $\Gamma_0\subset\Gamma$
and a single curve $C\in\Gamma$ such that $\Gamma'=\Gamma_0\cup\{C\}$ is hyperbolic.
But then there are only finitely many possibilities for $\Gamma'$,
since the shape of $\Gamma_0$ is limited by Lemmas \ref{lem:Dynkin} and 
\ref{lem:ortho}
while the Hodge Index Theorem gives
\begin{eqnarray}
\label{eq:C^2}
C^2 \leq \frac{(C.H)^2}{H^2} \leq \frac{d^2}{\tD} \leq \frac{d^2}2 
\end{eqnarray}
and
\begin{equation} \label{eq-bezout}
C.C' \leq d_C d_{C'} \;\;\; \forall C\, '\in\Gamma.
\end{equation}
Corollary \ref{cor:finite} thus gives an upper bound for $h$ when $\Gamma$ is hyperbolic.
For $h\gg 0$, $\Gamma$ therefore is parabolic, and Lemma \ref{lem:para} proves the claim.
\qed

\begin{rem}
For $\jD\gg 0$, 
this gives a quick proof of the bounds in  Theorems~\ref{thm}-\ref{thm3}
based on the results from Section \ref{s:para}.
Previously, this has been made effective only for the case $d=1$ 
over the field $k=\CC$, see \cite{degt}.
\end{rem}

\section{Preparations for the hyperbolic case}

In this section, we explain how  the above ideas lead to an effective constraint on the degree of the K3 surface in the hyperbolic case.  
Throughout this section, we assume that
\[
h>42d^2
\]
as in Theorem \ref{thm} (i), \ref{thm2} (i) and \ref{thm3} (i).
Certainly \eqref{eq:C^2} implies
\begin{eqnarray}
\label{eq:02}
C^2 \in\{0, -2\} \;\;\; \forall \, C\in\Gamma,
\end{eqnarray}
and \eqref{eq-bezout} can be improved drastically to $C.C'\leq d_Cd_{C'}/h$ in case $C^2, C'^2\geq 0$.
Thus we get
\begin{eqnarray}
C.C'=0 \;\;\; \forall \text{ isotropic } C, C'\in\Gamma,
\end{eqnarray}
so any two such curves are fibres of the same genus one fibration (given by $|C|=|C'|$),
and as such they are linearly equivalent.

We can directly extend these ideas to isotropic divisors.
Given an effective (thus nef) isotropic divisor $D\neq 0$ with $\deg(D)\leq 6d$, we 
claim that
\begin{eqnarray}
\label{eq:isotropic}
\label{eq:6d}
D.C=0 \;\;\; \forall \; C \in\Gamma.
\end{eqnarray}
To see this, consider the Gram matrix of 
$D$, $C$ and $H$. 
By the Hodge Index Theorem, its determinant is non-negative 
(cf. the proof of Lemma~\ref{lem:embed}), whereas 
\eqref{eq:02} and the main assumption of this section continue to hold. 
This implies directly that $D.C$ vanishes (for another argument, see  Example~\ref{example-7-4}).
 
 Similarly, one has
\begin{eqnarray}
\label{eq:<=2}
C.C' \leq 2 \;\;\; \forall \; C, C'\in\Gamma\;\;  \text{ with } C^2=C'^2=-2.
\end{eqnarray}

Recall that by a divisor $D$ of Kodaira type, we mean a nodal or cuspidal curve of arithmetic genus one
or a configuration of smooth rational curves (with multiplicities) which appears as a singular fibre of some elliptic surface.
Given a divisor $D=\sum_i n_i C_i$ of Kodaira type, one defines its weight 
 \[
 \mbox{wt}(D) := \sum_i n_i.
 \]
 In practice, wt$(\IK_n)=n$, wt$(\IK_n^*)=2n+6$, wt$(\IV^*)=12$, wt$(\III^*)=18$, wt$(\II^*)=30$. 
 Note that, if $D$ is supported on $\Gamma$, then 
 \begin{equation} \label{eq-weight-inequality}
 \deg(D)\leq \mbox{wt}(D)d \, .
 \end{equation}

The following result will prove very useful in the sequel:

\begin{lemm}
\label{lem:Kodaira_type}
If $\Gamma$ is not elliptic, then it supports a divisor of Kodaira type.
\end{lemm}


\begin{proof}
If there is an isotropic $C\in\Gamma$, we're done, so we may assume 
that $C^2=-2$ for all $C\in\Gamma$ by \eqref{eq:02}.
Then the statement follows from the classification of Dynkin diagrams:
any simple graph that is not a Dynkin diagram contains an extended one.
Its fundamental cycle gives the claimed divisor of Kodaira type.
\end{proof}

\begin{rem}
The statement of \ref{lem:Kodaira_type} can also be verified directly in our situation.
To this end, pick a (minimal) $\Gamma_0\subset\Gamma$ which is elliptic,
together with a curve $C_0\in\Gamma$ such that $\Gamma_0\cup\{C_0\}$
is not elliptic anymore.
If there is $C\in\Gamma$ such that $C_0.C=2$, then $C_0+C$ has Kodaira type $\IK_2$ or $\III$.
Otherwise, $C_0.C\leq 1$ for all $C\in\Gamma$ by \eqref{eq:<=2}.
By what we have seen before, $\Gamma_0$ is an orthogonal sum of root lattices,
and an easy case-by-case analysis confirms the claim.
\end{rem}

The consequence for the hyperbolic case is immediate: 
\begin{cor}
\label{cor:hyperbolic-isotropic}
If $\Gamma$ is hyperbolic, then it supports a divisor $D$ of Kodaira type,
and any such divisor has degree $\deg(D)>6d$.
In particular, there are no divisors of Kodaira types $\IK_n \, (n\leq 6), \II, \III, \IV, \IK_0^*$
supported on $\Gamma$,
and all curves in $\Gamma$ are smooth rational.
\end{cor}


\begin{proof}
The existence of $D$ follows from Lemma \ref{lem:Kodaira_type}.
But then $|D|$ induces a genus one fibration, 
and if $\deg(D)\leq 6d$, then any $C\in\Gamma$ is a fibre of this fibration by \eqref{eq:isotropic}.
Hence $\Gamma$ is parabolic, and we obtain a contradiction.

The statement about the Kodaira types follows from the degree bound \eqref{eq-weight-inequality}
in terms of the weight.
\end{proof}

\begin{rem}
\label{rem:hyper}
Corollary \ref{cor:hyperbolic-isotropic} implies that in the hyperbolic case,
one has $S_d=S_d'$, so
we do not have to limit ourselves to the restricted count in the exceptional
cases from Theorems \ref{thm2}, \ref{thm3}.
\end{rem}

Another  restriction on the possible hyperbolic graphs $\Gamma$
comes from considering hyperbolic subgraphs $\Gamma_0$ (cf.\ Remark \ref{rem:Gamma_0}).
Arguing with 
\[
L_0 = \ZZ\Gamma_0/\ker \hookrightarrow \Pic(X),
\]
one shows as in Proposition \ref{prop:h^2} (see \eqref{ineq-orthdecomp}) that the intrinsic polarisation $H_0\in L_0\otimes\QQ$
(if it exists) satisfies
\begin{eqnarray}
\label{eq:2h}
2h\leq H_0^2.
\end{eqnarray}

We illustrate the use of this bound by two examples, the second of which
will become important soon.

\begin{ex} \label{example-7-4}
If $0\neq D\in\Pic(X)$ is nef and isotropic, assume that $|D|$ admits some multisection $C\in\Gamma$.
Then applying the above argument to $\langle D,C\rangle$ exactly recovers the bound $\deg(D)>6d$ from \eqref{eq:6d}.
\end{ex}

Before coming to the second example, we introduce a general idea how to bound $H_0^2$ from above.
To this end, fix a basis of $L_0$ in $\Gamma$ with Gram matrix  $G$.
The intrinsic polarization 
\[
H_0=G^{-1}\vec d
\]
depends on the degree vector $\vec d$ (i.e.\ the coordinates $d_i$ of $\vec d$ 
are the degrees of the elements of the basis $L_0$),
and estimating $H_0^2$ may amount to a non-trivial optimization problem. However,
since the degrees are positive, there is a rough bound 
in terms of the entries $g_{ij}$ of $G^{-1}$ by
\begin{eqnarray} \label{eq-7-6}
\label{eq:H_0^2}
H_0^2 \leq \sum_{i,j} \max(0,g_{ij}) d^2.
\end{eqnarray}
\noindent
Indeed, we have $H_0^2 = H_0^{\transpose} GH_0 = {\vec d}^{\transpose} G^{-1} \vec d = \sum_{i,j} d_ig_{ij}d_j$, so \eqref{eq-7-6} follows from the inequalities
$\sum_{i,j} d_ig_{ij}d_j \leq \sum_{i,j}\max(0,g_{ij}) d_id_j \leq \sum_{i,j} \max(0,g_{ij})d^2.$

\begin{ex}
\label{ex:D_6}
Let $D$ be a divisor of Kodaira type $\IK_2^*$,
corresponding to an extended Dynkin diagram $\tilde D_6\subset\Gamma$.
Assume that 
there are 3 disjoint $(-2)$-curves of degree at most $d$ on $X$, 
serving as sections for the fibration induced by $|D|$,
and meeting different components of $D$.
The corresponding vertices in $\Gamma$ connect to different monovalent vertices of $\tilde D_6$.
This gives a rank $10$ hyperbolic lattice $L_0$.
For its Gram matrix $G$ one has $\sum_{i,j} \max(0,g_{ij})=78\frac 2{21}$, 
%
so \eqref{eq:2h} and \eqref{eq:H_0^2} contradict our assumption $h>42d^2$,
i.e.\ this configuration is impossible.
\end{ex}

%
%

\section{Proof for non-exceptional hyperbolic case}

We are now in the position to make our previous ideas effective.
To this end, in this section we make the following

\begin{assumption}
$\Gamma$ is hyperbolic with $\#\Gamma>24$.
\end{assumption}

Observe that for  characteristic $p \neq 2,3$  the first assumption follows from the 
second (by Corollary~\ref{cor:21} and \eqref{eq:<=24}).
Recall that by Corollary \ref{cor:hyperbolic-isotropic}, all curves in $\Gamma$ are smooth rational,
and as hinted in Remark \ref{rem:hyper}, we can treat all characteristics almost alike 
(see the next section for the few subtleties remaining).
Note also that 
\[
C.C'\leq 1 \;\;\; \forall C, C'\in\Gamma
\]
by \eqref{eq:<=2}, since the case $C.C'=2$ would lead to a divisor of Kodaira type $\IK_2$ or $\III$
of degree $\leq 2d$, contradicting Corollary \ref{cor:hyperbolic-isotropic}.

By Lemma \ref{lem:Kodaira_type},
$\Gamma$ 
supports a divisor $D$ of Kodaira type (with $\deg(D)>6d$ by Corollary \ref{cor:hyperbolic-isotropic}).
We proceed with two reduction steps.

\begin{lemm}[First reduction step]
\label{lem:red1}
{Given a divisor $D$ of Kodaira type supported on $\Gamma$,
assume that $|D|$ is not quasi-elliptic.
Then
there are at least 3 multisections of $|D|$ in $\Gamma$.}
\end{lemm}

\begin{proof}
Assume to the contrary that $\Gamma$ contains at most 2 multisections for $|D|$. 
In consequence, $\Gamma$ contains at least 23 fibre components.
Regardless of there being sections or not,
any orthogonal sum of root lattices $L_0$ embedding into the fibres
embeds into $\Pic(X)$ with orthogonal complement hyperbolic indefinite.
(This is just like in the jacobian case, where the orthogonal complement contains the hyperbolic plane $U$.)
In particular, 
\begin{eqnarray}
\label{eq:L_0}
\rank L_0 \leq 20.
\end{eqnarray}
In case there are at most two  singular fibres completely supported on $\Gamma$
(i.e.\ $\Gamma$ contains the corresponding extended Dynkin diagrams),
omitting a single component of each of these yields $L_0$ of rank $21$,
contradicting \eqref{eq:L_0}.
Thus $\Gamma$ has to support at least 3 singular fibres completely.
If there were four or more of them, then connecting  any multisection in $\Gamma$
with one fibre component in each fibre would produce 
a divisor of Kodaira type $\IK_0^*$ supported on $\Gamma$,
contradicting Corollary \ref{cor:hyperbolic-isotropic}. 
(This case can also be ruled out by considering the Euler--Poincar\'e characteristic.)
Hence there are exactly three  singular fibres completely supported on $\Gamma$,
and we get $\rank L_0=20$.
Comparing Euler--Poincar\'e characteristic and degree,
one of the divisors $D$ of these fibres has $\deg(D)\leq 12d$.
The analogue of \eqref{eq:6d} then shows that
\begin{eqnarray}
\label{eq:12d}
D.C \leq 1\;\;\; \forall \, C\in\Gamma.
\end{eqnarray}
That is, all curves in $\Gamma$ are either fibre components or sections of the fibration induced by $|D|$.
In particular, the fibration is jacobian, and by the Shioda--Tate formula, the rank of $L_0$
implies that $\MW(|D|)$ is finite.
These (quasi-)elliptic surfaces are called extremal, 
and they are very rare. 
Indeed, the classification of extremal elliptic K3 surfaces by Ito in \cite{Ito1}, \cite{Ito2}  reveals
that the only possibility with three singular fibres allowed by Corollary \ref{cor:hyperbolic-isotropic} has
configuration $\IK_7, \IK_7, \II^*$ and $\MW=\{O\}$, in characteristic $p=7$ only, so $\#\Gamma=7+7+9+1=
24$,
contradiction.
\end{proof}

\begin{rem}
\label{rem:quasi-ell}
In the quasi-elliptic case, 
there are a few further configurations with 3 singular fibres supported on $\Gamma$
(denoted by extended Dynkin diagrams)
and two sections only (see e.g. \cite[Table QE]{Shimada}):
\begin{enumerate}
\item
\label{item1}
$p=3$, 
configuration
$3\tilde E_6+A_2$
with two sections from $\MW(|D|) \cong \ZZ/3\ZZ$;


\item
\label{item2}
$p=2$, 
configuration $3\tilde D_6+2A_1$ with $\MW(|D|) \supset \ZZ/2\ZZ$;

%

\item
\label{item3}
$p=2$, 
configuration $2\tilde E_7+\tilde D_6$ with $\MW(|D|)=\ZZ/2\ZZ$.
\end{enumerate}
There are two other configurations which are a priori possible by \cite[Table QE]{Shimada},
$2\tilde E_6+\tilde E_8$ and $3\tilde E_6+A_2$,
but both have $\MW=\{O\}$,
so $\#\Gamma\leq 24$.
\end{rem}

In the proof of  the second reduction step below,
the notion of weight of a divisor of Kodaira type and the inequality  \eqref{eq-weight-inequality}
again play an important role.


\begin{lemm}[Second Reduction step]
\label{lem:red2}
$\Gamma$ supports a divisor of Kodaira type $\IK_1^*, \IK_2^*$ or $\IV^*$.
\end{lemm}

\begin{proof}
By Corollary \ref{cor:hyperbolic-isotropic}, $\Gamma$ supports a divisor $D$ of Kodaira type.
Assume that $D$ does not have Kodaira type $\IK_1^*, \IK_2^*$ or $\IV^*$.
This means that the exceptional configurations from Remark \ref{rem:quasi-ell} cannot occur,
so by Lemma \ref{lem:red1}, the fibration induced by $|D|$ has at least 3 multisections in $\Gamma$.
If one of them were not a section, 
 then either it would meet two irreducible components of $D$,
thus giving a cycle of weight less than $\mbox{wt}(D)$, or it would meet a multiple component of $D$.
For $\IK_n^* \, (n>2)$, this results in a divisor of type $\IK_m^*$ with $m<n$,
while $\II^*$ may also give $\III^*$ or $\IV^*$, and $\III^*$ may also give $\IV^*$.
In any case, the weight drops or we get one of the stated types.
So we may assume that all non-fibre components in $\Gamma$ are sections
of the fibration induced by $|D|$
and proceed with a case-by-case analysis.

{If $D$ has type $\II^*$}, 
then the 3 sections meet one and the same (simple) component,
and we would get $\IK_3$ (if two of them meet) or $\IK_0^*$,
both of which are excluded by Corollary \ref{cor:hyperbolic-isotropic}.

In what follows we will often suppress those cases ruled out 
in the same fashion.

{If $D$ has type $\III^*$}, the 3 sections lead to cases as above
or we get $\IK_3^*$ of weight $12<18=\mbox{wt}(\III^*)$.

{If $D$ has type $\IK_n^*\, (n>2)$},
we get  $\IK_1^*$ or $\IV^*$
(since two of the sections meet simple fibre components which connect through a single fibre component,
and they are disjoint by Corollary \ref{cor:hyperbolic-isotropic}).

{If $D$ has type $\IK_n \, (n>6)$},
we either get a cycle
of length at most $\frac n3+3<n$ or a divisor of type $\IK_m^*$ 
(the precise value $m\leq \frac n3$ does not matter)
to which we then apply the previous step.
\end{proof}



\subsection{Proof of Theorems \ref{thm} (i), \ref{thm2} (i) and \ref{thm3} (i)}
\label{ss:pf}

We continue to assume $h>42d^2$.
By Lemma \ref{lem:red2},
it  remains to rule out  configurations with $\#\Gamma>24$ involving a divisor $D$ of Kodaira type
$\IK_1^*, \IK_2^*$ or $\IV^*$ inducing an elliptic fibration
 -- like we already did for
a special configuration in Example \ref{ex:D_6}.

If $D$ has type $\IK_1^*$, then the 3 sections either lead to a cycle of length at most 6 (which is impossible
by Corollary \ref{cor:hyperbolic-isotropic}),
or the sections are disjoint, and together they support a divisor $D'$ of type $\IK_2^*$ or $\IV^*$,
but with one component of $\IK_1^*$
now forming a bisection of the fibration induced by $|D'|$,
contradicting \eqref{eq:12d}.


If $D$ has type   $\IK_2^*$, we distinguish whether the 3 sections are pairwise disjoint.
In the disjoint case, two of the sections have to meet the same component
(since otherwise Example \ref{ex:D_6} gives a contradiction),
so $\Gamma$ supports $\IK_1^*$ as well, and the previous case gives a contradiction.
If the sections are not all disjoint, then $\IK_2^*$ extended by two sections which meet
supports a divisor $D'$ of type $\III^*$
while the remaining fibre component of $\IK_2^*$ serves as a trisection of
the fibration induced by $|D'|$.
Since $\deg(D')\leq 18d$, this contradicts the analogue of \eqref{eq:6d}, \eqref{eq:12d}.


If $D$ has type  $\IV^*$, 
either  some of the sections meet, giving a configuration of type $\IK_3$ which we ruled out before, 
or $\IK_7$ which, with 3 sections attached, leads on to $\IK_1^*$ or $\IK_2^*$,
or the sections are disjoint. 
The latter case leads to $\IK_2^*$ (so that the previous considerations give a contradiction) or 
to a single central vertex extended by 3 disjoint $A_3$ configurations.

For this last case,  we analyse more closely the possible configurations in $\Gamma$.
Namely, if there are more than 3 sections, then we are automatically in the cases with $\IK_7$ or $\IK_2^*$ above,
so we may assume that $\Gamma$ contains exactly three sections of
the fibration induced by $|D|$.
Hence there are at least 22 fibre components in $\Gamma$.

If there are less than 3 fibres supported completely on $\Gamma$,
then we  again obtain an extremal fibration
which only leads to exceptional cases (see Section \ref{ss:adjust}).

Assume that there are 3 fibres completely supported on $\Gamma$, one of them being $D$.
Here any elliptic configuration has an $\IK_n$ fibre supported on $\Gamma$
 (with $n>6$ by Corollary \ref{cor:hyperbolic-isotropic}),
since else the Euler--Poincar\'e characteristic reveals 
that there can be at most 21 fibre components in $\Gamma$.
But then any given section connects with $\IK_n$ and the other two fibres supported on $\Gamma$ 
to give a divisor of type $\IK_1^*$ as treated before. 

In summary, if $h>42d^2$ and we are outside the exceptional cases, then 
no K3 surface of degree $2h$ admits
a configuration $\Gamma$ of more than 24 rational curves of degree at most $d$ 
such that $\Gamma$ is hyperbolic. 
Thus, by Corollary~\ref{cor:21} and \eqref{eq:<=24}, we obtain
$$
S_d \leq 24
$$
which completes the proof 
of  Theorems \ref{thm} (i), \ref{thm2} (i) and \ref{thm3} (i).
\qed

%
%

\section{Proofs of  Theorems \ref{thm2}  and \ref{thm3}}
\label{s:exc}

\label{ss:adjust}

To complete the proofs of  Theorems \ref{thm2} and \ref{thm3},
it remains to cover the exceptional cases in characteristics $2$ and $3$.
In the argument from \ref{ss:pf}, there are the following exceptional cases occurring
(continuing the numbering from Remark \ref{rem:quasi-ell}).
First with less than 3 fibres completely supported on $\Gamma$:
\begin{enumerate}
\item[(4)]
\label{item4}
$p=2$, elliptic fibration with configuration $\tilde A_{11}+\tilde E_6+A_3$
and $\MW=\ZZ/3\ZZ$ by \cite{Ito1}, \cite{Ito2};
\item[(5)]
\label{item5}
$p=3$, quasi-elliptic fibration with configuration
  $2\tilde E_6+E_6+A_2$ and $\MW=\ZZ/3\ZZ$.
\end{enumerate}
The only other a priori possible configuration   $2\tilde E_6+4A_2$
($p=3$, quasi-elliptic) is ruled out as follows.
The three sections are 3-torsion (as enforced by quasi-elliptic fibrations).
Hence, 
for the height pairing from \cite{MWL} to evaluate as zero, any two of them have to meet exactly three out of 
the four $\IV$ fibres in different components. In particular, this implies that some section
meets two of the $A_2$ summands. But then it connects with these two and with the two $\IV^*$ fibres
to a divisor of type $\IK_0^*$, so we obtain a contradiction (see Corollary \ref{cor:hyperbolic-isotropic}).

If there are 3 fibres completely supported on $\Gamma$, one of them being $D$,
then the quasi-elliptic case only allows for 
\begin{enumerate}
\item[(6)]
\label{item6}
$p=3$, configurations $3\tilde E_6+A_1$ or $3\tilde E_6+A_2$ 
as in \eqref{item1}, but now with all three sections contained in $\Gamma$.
\end{enumerate}

\subsection{Degree bound $h>43d^2$ in characteristic $3$}

One easily verifies that each exceptional case in characteristic $p=3$
(i.e.  \eqref{item1}, (5), (6))
features a divisor of Kodaira type $\IK_3^*$ with 4 disjoint sections,
one meeting each simple fibre component
(the monovalent vertices in the corresponding extended Dynkin diagram).
Let $G$ denote the Gram matrix.

\begin{lemm}
In the box $[0,d]^{12}$, the product $\vec x G^{-1}\vec x^\top$ is
maximized
by $\vec{x}_{max}=(d,\hdots,d)$.
\end{lemm}

\begin{proof}
The inverse $G^{-1}$ of the Gram matrix $G$ has only a few negative entries,
occurring in $2\times 2$ blocks of the shape $A=\frac 13\begin{pmatrix}-2 & -1\\-1 & -2\end{pmatrix}$.
We define the auxiliary matrix 
\[
G_0 = 
\begin{pmatrix} 
A & 0 & 0 & 0 & -A & 0\\
0 & A & 0 & 0 & 0 & -A\\
0 & 0 & 0 & 0 & 0 & 0\\
0 & 0 & 0 & 0 & 0 & 0\\
-A & 0 & 0 & 0 & A & 0\\
0 & -A & 0 & 0 & 0 & A
\end{pmatrix}
\]
where each entry stands for a $2\times 2$ block.
This results in the decomposition
\[
G^{-1} = G_0 + G_+
\]
where all entries of $G_+$ are non-negative.
Moreover $G_0$ is negative-semidefinite with $\vec x_{max}$ in its kernel,
so $\vec x_{max}$ maximizes $\vec x G_0\vec x^\top$.
Obviously, it also optimizes $\vec x G_+\vec x^\top$,
and the claim follows.
\end{proof}

\subsection*{Proof of Theorem \ref{thm3} (ii)}

Let $\Gamma_0\subset\Gamma$ be given by the 12 smooth rational curves in the above configuration (i.e. a divisor of Kodaira type $\IK_3^*$ with 4 disjoint sections,
one meeting each simple fibre component).
We estimate the square of the intrinsic polarization $H_0$.
Since $G_0$ has zero sum of entries, $G^{-1}$ and $G_+$ have the same sum of entries $86$.
Arguing as in \eqref{eq:H_0^2}, we find $H_0^2\leq 86d^2$,
so this configuration is excluded as soon as $h>43d^2$.
\qed

%
%
%
%
%
%
%
%
%
%
%
%
%
%
%

\subsection{Degree bound $h>46.25d^2$ in characteristic $2$}

Each exceptional case in characteristic $p=2$
(i.e.  \eqref{item2}, \eqref{item3}, (4))
features a divisor of Kodaira type $\IV^*$ extended by three disjoint $A_2$ configurations
(or a single central vertex extended by 3 disjoint $A_4$ configurations).
The inverse of the Gram matrix  has few negative entries and sum of entries $92.5$.
The proof of Theorem \ref{thm2} (ii) is similar to the above;
the details are left to the reader.

\section{K3 surfaces with 24 rational curves}
\label{s:ex}

In this section, we prove Theorem \ref{thm} (iii) fixing $d\geq 3$.
We need the following auxiliary result.

\begin{lemm}
\label{lem:rk2}
Assume that $\operatorname{char}(k)\neq 2$
and let $c\in2\ZZ$.
Then there is a family of K3 surfaces over $k$ with generic Picard lattice
\[
\Pic = \begin{pmatrix}
0 & d \\
d & c
\end{pmatrix}.
\]
\end{lemm}

\begin{proof}
We will obtain the desired K3 surfaces by deforming certain other K3 surface $Y_r$.
To set up the K3 surfaces, fix $c_0\in\{2,4,\hdots,2d\}$ such that
\[
c\equiv -c_0\mod 2d.
\]
Write $c_0-2$ as twice the sum of at most $r$ squares:
\[
c_0 = 2(1+n_2^2+\hdots+n_r^2), \;\;\; 1\leq r\leq 5, \; n_i\in\NN.
\]
Consider a K3 surface $Y_r$
admitting an elliptic fibration with zero section $O$
and $r$ fibres of type $\IK_2$.
Independent of the characteristic,
it is a consequence of Tate's algorithm \cite{Tate}
that $Y_r$ can be given in Weierstrass form
\[
y^2 + a_2 xy + a_{6-r}g_ry = x^3 + a_4x^2 + a_{8-r}g_rx+b_{12-2r}g_r^2
\]
where $a_i, b_j, g_r\in k[t]$ with the subscript indicating the degree
and $g_r$ squarefree.
One  verifies that the family of such elliptic K3 surfaces depends on $(18-r)$ parameters,
so the generic member will have $\rho=2+r$.
Here we shall work with a
general member $Y_r$ of this family which has no other reducible singular fibres
while being non-supersingular.
Thus there is a primitive embedding
\[
U \oplus A_1^r\hookrightarrow \Pic(Y_r) 
\]
where the sublattice is
generated by the fibre $F$, the zero section $O$ and the fibre components $\Theta_1,\hdots,\Theta_r$
not meeting $O$
(additional generators of $\Pic(Y_r)$ are given by sections by \cite{MWL}).
Let $N\in\NN$ and consider the divisor
\[
H = NF + dO - \Theta_1-n_2\Theta_2-\hdots-n_5\Theta_5.
\]
For $N>2d$, one directly verifies using Criterion \ref{crit}
that $H$ is very ample.

%
By \cite[Prop. 1.5]{Deligne}, the K3 surface $Y_r$ deforms together with the divisor classes $H, F$
in an 18-dimensional family over $k$.
By construction, the generic member has Picard lattice isometric to the stated one.
\end{proof}

\begin{rem} For complex K3 surfaces Lemma~\ref{lem:rk2} follows immediately from the general theory
developed by Nikulin (\cite[Cor.~1.12.3]{nikulin}, see also \cite[Cor.~14.3.1]{Huybrechts}). 
Over fields of positive characteristic, however, lattice theory no longer suffices to prove the existence
of K3 surface with a given Picard group  (see e.g. \cite[Remark~14.3.2]{Huybrechts}).
\end{rem}

\subsection*{Proof of Theorem \ref{thm} (iii)}
We continue with the family of K3 surfaces from  Lemma \ref{lem:rk2}.
Applying an isometry of $\Pic$, we may assume that $c\in\{-2,0,2,\hdots,2d-4\}$.
Denote the generators of $\Pic$ by $D, C$.
By Riemann--Roch, the isotropic vector $D$ is either effective or anti-effective,
so let us assume the former by adjusting the signs of $D$ and $C$, if necessary.
Then $|D|$ may still involve some base locus
which can be eliminated by the composition $\sigma$ of a finite number of reflections.
The resulting divisor $E=\sigma(D)$ is a fibre of a genus one fibration.
We choose $X$ general in the family of K3 surfaces
such that $X$ is not supersingular and all singular fibres of the genus one fibration are nodal cubics
(so they are 24 in number by \eqref{eq:e=24}).

Turning to the divisor $B=\sigma(C)$, it is effective, again by Riemann--Roch (and since $E.B=d>0$).
Since the fibre class $E$ is nef, every irreducible component $B'$ of the support of $B$ satisfies
\[
0 \leq B'.E\leq d,
\]
where the left inequality becomes equality if and only if $B'$ itself is a fibre.
Arguing with all components of the support of $B$ and with all other multisections of the genus one fibration,
one verifies using Criterion \ref{crit} that $NE+B$ is very ample for $N>2d$.


Applying this procedure separately to all values $c\in\{-2,0,2,\hdots,2d-4\}$,
we find   K3 surfaces of degree $\tD$ with 24 rational curves of degree exactly $d$
(the images of the singular fibres)
for all 
$\jD\geq d(2d+1)-1$.
\qed

\begin{rem}
\label{rem:24-2}
In characteristic $2$, the same construction can be carried out 
to produce K3 surfaces with ample divisors $H$ by the Nakai--Moishezon criterion.
Then $mH$ is very ample for all $m\geq m_0$ for a certain $m_0\in\NN$,
so we get projective models of K3 surfaces with 24 rational curves of degree $md$
(and infinitely many such polarizations fixing $md$, 
because we can always add positive multiples of the nef divisor
$E$ to $mH$ to obtain further very ample divisors).
\end{rem}


\section{K3 surfaces with 24 smooth rational curves}
\label{s:smooth}

This section aims to exhibit explicit projective models of K3 surfaces attaining the bounds
from Theorems \ref{thm} -- \ref{thm3} for infinitely many 
degrees $H^2=\tD$
although we restrict to smooth rational curves exclusively.
Throughout we fix an integer 
 $d$ and only consider
K3 surfaces with smooth rational curves of degree exactly $d$
to simplify the exposition.
By the discussion of the hyperbolic case in  \ref{ss:pf}, once $\jD>42d^2$,
all curves have to be fibre components of some genus one fibration.
In the non-quasi-elliptic case (e.g. outside characteristics $2,3$),
comparing \eqref{eq:e=24} and \eqref{eq:comps} shows that
all singular fibres have to be multiplicative and reducible (Kodaira type $\IK_n, n>1$);
since $F.H=dn$ is fixed, they \emph{all} have the same type.
This gives three cases,
\begin{eqnarray}
\label{eq:configs}
12\times \IK_2,\;\;\; 8\times \IK_3,\;\;\; 6\times \IK_4
\end{eqnarray}
which we will study in detail in what follows.
(There is an additional case $4\times \IK_6$ in characteristic 2
while the other combinatorial cases are ruled out by the Shioda--Tate formula \cite[Cor. 5.3]{MWL}.)

We shall start with models covering the minimal degree $d=1$.
This obviously rules out the first configuration from \eqref{eq:configs}
since then each pair of fibre components would meet in two points
which is impossible for lines.

\subsection{Fermat surface ($6\times \IK_4$)}

Assume that char$(k)\neq 2$.
Let $X_4$ be the Fermat quartic surface,
defined by
\[
X_4 = \{x_0^4-x_1^4+x_2^4-x_3^4=0\}\subset\PP^3.
\]
This has 48 lines over $k$  
 (112 in characteristic $3$, see 
\ref{ss:3}),
and the signs were chosen for 8 lines  such as
\begin{eqnarray}
\label{eq:4lines}
\{x_0\pm x_1 = x_2\pm x_3=0\}
\end{eqnarray}
to be defined over the prime field.
As noted in \cite{Barth}, the morphism
\begin{eqnarray*}
\pi: \;\;\;\;\;\;\;\; X_4 \;\;\;\;\;\;\; & \to & \;\;\;\;\;\;\;\;\;\PP^1\\
~[x_0,x_1,x_2,x_3] & \mapsto & [x_0^2-x_1^2,x_2^2+x_3^2]
\end{eqnarray*}
defines a genus one fibration with 6 fibres of Kodaira type $\IK_4$,
each comprising four of the lines
(for instance, the fibre at $[0,1]$ is just the 4-cycle of lines from \eqref{eq:4lines}).
The other lines serve as bisections,
and over $\CC$ or fields of characteristic $p\equiv 1\mod 4$,
one can show that there are no sections. 
The next fact, just like the ones to follow, can easily be checked using Criterion \ref{crit},
so we omit the details.

\begin{Fact}
Let $F$ denote a fibre of $\pi$,
$H_0$ a hyperplane section of $X_4\subset\PP^3$ and $N\in \NN_0$. Then $H=NF+dH_0$ is very ample.
\end{Fact}


For any $\jD=2d(2N+d)$ we thus obtain a degree-$\tD$  model of $X_4$
containing 24 smooth rational curves of degree $d$
(the images of the fibre components).
\begin{rem}
In characteristic $p\equiv-1\mod 4$, $X_4$ is supersingular,
and the fibration $\pi$ has sections, accounting for the jump of the Picard number.
The sections allow us to find  polarizations of $X_4$
for further values for $H^2$
with 24 smooth rational curves of degree $d$.
\end{rem}

\subsection{Hesse pencil ($8\times \IK_3$)}
Assume that char$(k)\neq 3$.
Let $S$ be the rational elliptic surface defined by the Hesse pencil
\[
S: \;\;\; x^3+y^3+z^3 = 3txyz.
\]
Then $S$ has 4 singular fibres of type $\IK_3$ at $\infty$ and the third roots of unity.
The Mordell--Weil group of $S$ consists of the nine base points of the pencil, $P_1,\hdots,P_9$.

\begin{lemm}
Let $F'$ denote a fibre of $S$. Then the class $D'=F'+P_1+\hdots+P_9$ is  3-divisible in $\Pic(S)$.
\end{lemm}

\begin{proof}
%
Picking $P_9$ as zero of the group law, say, 
the theory of Mordell--Weil lattices  \cite{MWL} gives an isomorphism
\[
\MW(S) = \Pic(S)/(\text{trivial lattice generated by fibre components and } P_9).
\]
Presently this yields
\[
P_1 + \hdots + P_8 = 8P_9 + 3 \sum_{i=1}^4 \Theta_i - 4 F',
\]
where the $\Theta_i$ denote the rational fibre components met by $P_9$.
Adding $P_9$ and $F'$ on both sides, we derive the claimed divisibility.
\end{proof}

Consider the base change $X_3$ of $S$ by a separable quadratic morphism
$\PP^1\to\PP^1$
which is unramified at the singular fibres.
Then $X_3$ is an elliptic K3 surface with eight fibres of type $\IK_3$ and the said nine sections
(but now featuring as $(-2)$-curves).
Let $F$ denote a fibre of $X_3$.
By pull-back, the divisor $D=2F+P_1+\hdots+P_9\in\Pic(X_3)$ is 3-divisible.

\begin{Fact}
Let $N\in\NN$. Then $H=NF+\frac d3 D\in\Pic(X_3)$ is very ample.
\end{Fact}


This gives degree-$\tD$ models of $X_3$ with 24 smooth rational curves
of degree $d$
for $\jD=d(3N+d)$. 

\begin{rem}
\label{rem:24-22}
As in Remark \ref{rem:24-2}, for infinitely many $d$, we  obtain non-unirational projective K3 surfaces
with infinitely many polarizations of degree $\tD$,
containing exactly 24 smooth rational degree-$d$ curves in characteristic 2.
\end{rem}

%
%
%

\subsection{Configuration $12\times \IK_2$}

Assume that char$(k)\neq 2$ and consider squarefree polynomials $f,g\in k[t]$
of degree $4$ such that $f-g$ is also squarefree of the same degree.
Then the extended Weierstrass form 
\[
y^2 = x(x-f)(x-g)
\]
defines an elliptic K3 surface $X_2$ over $\PP^1_t$
with twelve singular fibres of type $\IK_2$ at the zeroes of $f, g$ and $f-g$.
Generically, one has $\MW(X_2)\cong(\ZZ/2\ZZ)^2$
with disjoint sections 
\[
P_1 = (0,0), \;\;\; P_2 =  (0,f),\;\;\; P_3 =  (0,g)
\]
and $P_0$ the point at $\infty$.

\begin{Fact}
Assume that $d$ is even.
Let $F$ denote a fibre and $N>d$. Then
$H=NF+\frac d2 (P_0+P_1+P_2+P_3)$ is very ample.
\end{Fact}

Therefore, we obtain  
 degree-$\tD$ models of $X_2$ for $\jD=d(2N-d)$
(and $d$ even)
with 24 smooth rational curves of degree $d$
in characteristic $\neq 2$.

\begin{rem}
Assuming $X_2$ to arise from a rational elliptic surface with six fibres of type $\IK_2$,
one can endow $X_2$ with two additional independent sections. 
This also allows one to realize polarizations $H^2\equiv d^2\mod 4d$.
\end{rem}

\subsection{Extra bound in characteristic 3}
\label{ss:p=3}\label{ss:3}

The Fermat quartic $X_4\subset\PP^3$
admits further elliptic fibrations;
one can be obtained by
fixing any  line $\ell\subset X_4$
and considering the pencil of hyperplanes containing $\ell$.
In characteristic $3$, the resulting fibration is quasi-elliptic with 10 fibres of type $\IV$
(compare \cite{RS-112} for the special role of this surface among quartics in characteristic $3$).
As before, denote a fibre of the fibration in question by $F$.

\begin{Fact}
If $N>2d/3$, then the divisor $H=NF+d\ell$ is very ample.
\end{Fact}

Thus we obtain degree-$\tD$ models of $X_4$ with $\jD=3dN-d^2$
which contain 30 smooth rational curves of degree $d$ in characteristic 3.

\subsection{Bounds in characteristic 2}
\label{ss:2}

We have already seen in Remarks \ref{rem:24-2}, \ref{rem:24-22}
how the bound from Theorem \ref{thm2} (i) can be attained in characteristic 2.
It remains to establish the same statement for the bound in 
Theorem \ref{thm2} (ii).
To this end, let $d\geq 2$ and consider a quasi-elliptic K3 surface $X$
with 20 fibres of type $\III$ as in \cite{RuS}.
Then the curve of cusps $C$ can be regarded as a smooth rational bisection
which meets each component of a reducible fibre with multiplicity one.
Denoting a fibre by $F$, it follows from the Nakai--Moishezon criterion
that the divisor $H=NF+dC$ is  ample for any $N>d$.
%

For $m\gg 0$, we thus obtain projective models of $X$ with infinitely many different polarizations
$mH+nF \;(n\in\NN_0)$, 
each of which contains exactly 40 smooth rational curves of degree $md$.

\smallskip
The bound \eqref{eq-25} from Remark \ref{rem} arises from considering
quasi-elliptic K3 surfaces $X'$ with 5 fibres of type $\IK_0^*$.
We conjecture that  \eqref{eq-25} is sharp (at least for large $h$). Indeed, let us consider the surface  $X'$.
The curve of cusps $C$ and the double fibre components $\Theta_1,\hdots,\Theta_5$
allow us to define the ample divisor
\[
H = NF + d(3C + \Theta_1+\hdots+\Theta_5)
\]
for $N>d/2$. Since $H$ meets all the requirements of Criterion \ref{crit} for $N\geq 2$, 
we conjecture that, at least for $N\gg 0$, it is very ample
and thus yields projective models of $X'$ with 25 lines.


\subsection*{Acknowledgement} 
We are grateful to the anonymous referee for valuable comments.
S.~Rams would like to thank J.~Byszewski for inspiring remarks.

%
%
%
%
%
%
%

\end{document}